\pdfoutput=1
\documentclass[11pt]{amsart}
\usepackage{amsmath,amssymb,amsthm}
\usepackage[margin=1.1in]{geometry}
\usepackage{hyperref}
\hypersetup{pdftitle={Proofs of some OEIS conjectures on Wythoff sums, Fibonacci and Lucas words},pdfauthor={Alex Ashburn}}
\usepackage{booktabs}

\newtheorem{theorem}{Theorem}

\newtheorem{proposition}[theorem]{Proposition}

\theoremstyle{definition}

\theoremstyle{remark}
\newtheorem{remark}[theorem]{Remark}

\newcommand{\fl}[1]{\left\lfloor #1\right\rfloor}
\newcommand{\N}{\mathbb{N}}

\title[Conjectures on Wythoff sums and the Lucas word]{Proofs of some OEIS conjectures on Wythoff sums,\\ Fibonacci and Lucas words}

\author{Alex Ashburn}
\address{Independent Researcher}
\email{aashburn1998@gmail.com}


\begin{document}

\begin{abstract}
We prove several conjectures from the On-Line Encyclopedia of Integer Sequences about the lower and upper Wythoff sequences and the Fibonacci word.
Among them are two of Kimberling's three conjectures on the number of ways to write $n=\lfloor h\varphi\rfloor+\lfloor k\varphi^2\rfloor$ with $h,k\ge1$ (A259598): exactly one way if and only if $n+1=2F$ for a Fibonacci number $F\ge2$, and exactly two ways if and only if $n+1\ge7$ is a Lucas number.
The third conjecture, that no way exists if and only if $n+1$ is a Fibonacci number, was proved earlier by Kawsumarng et al.
We also observe that Kimberling's conjecture on the gaps of the sums of two distinct terms of A003622 and of their complement (A333308, A333309) follows, after a shift by $2$, from earlier \texttt{Walnut} results of Shallit and of Bosma et al.\ on A260317, and we re-verify it.
Next, we prove Kimberling's five 2025 conjectures on the gaps between positions where the Fibonacci word and the ``Lucas word'' take prescribed values (A383423--A383427).
Finally, we prove a conjecture of Mathar on A285383, and we point out that a conjecture of Schmidt on A003250 follows from theorems of Carlitz, Scoville and Vaughan (1973); we also confirm it with \texttt{Walnut}.
Most proofs are decision procedures run in the free prover \texttt{Walnut}, and we supply the complete command file.
\end{abstract}

\maketitle

\section{Introduction}

Let $\varphi=(1+\sqrt5)/2$.
The lower and upper Wythoff sequences are $u(n)=\fl{n\varphi}$ (A000201) and $v(n)=\fl{n\varphi^2}=u(n)+n$ (A001950).
The Fibonacci word $\mathbf f=f_0f_1f_2\cdots=0100101001001\cdots$ (A003849) is the fixed point of $0\mapsto01$, $1\mapsto0$.
The On-Line Encyclopedia of Integer Sequences~\cite{oeis} records many conjectures about these objects, several of them due to Clark Kimberling.
In this note we prove the following, which were listed as conjectures in the OEIS as of September 2026.

\begin{enumerate}
\item \textbf{A259598} (Kimberling, 2015). Let $r(n)$ be the number of pairs $(h,k)$ with $h,k\ge1$ and $n=u(h)+v(k)$. Kimberling conjectured the following:
  \begin{enumerate}
  \item $r(n)=0$ iff $n\in\{1,2,4,7,12,\dots\}=\{F-1:F\ge2 \text{ a Fibonacci number}\}$;
  \item $r(n)=1$ iff $n\in\{3,5,9,15,25,41,\dots\}$ (the terms $\ge3$ of A001595);
  \item $r(n)=2$ iff $n\in\{6,10,17,28,46,\dots\}$ (the terms $\ge6$ of A001610).
  \end{enumerate}
  Part (a) was proved by Kawsumarng, Khemaratchatakumthorn, Noppakaew and Pongsriiam~\cite{kawsumarng}, and reproved with \texttt{Walnut} by Shallit~\cite{shallit-sumsets}. We prove (b) and (c).
\item \textbf{A333308, A333309} (Kimberling, 2020). Let $S$ be the set of sums of two distinct terms of A003622 $=\{v(n)-1:n\ge1\}$. Conjecture: the difference sequences of $S$ and of its complement $\{1,2,3,\dots\}\setminus S$ consist of Fibonacci numbers. The gap sets are exactly $\{1,2,3\}$ and $\{1,2,3,5\}$. This is not new: $S+2$ is the set of sums of two distinct upper Wythoff numbers, whose complement is A260317, so the statement follows from Shallit's \texttt{Walnut} proof that the gaps of A260317 are exactly $1,2,3,5$~\cite{shallit-cicm} and from a lemma of Bosma et al.~\cite{bosma}; see Section~\ref{sec:333308}.
\item \textbf{A383423--A383427} (Kimberling, 2025). Five conjectures on the gaps between the positions $k$ at which $(f_k,\ell_k)$ takes a given value, where $\boldsymbol\ell$ is the Lucas word A383422.
\item \textbf{A003250} (Schmidt, 2014): $a(n)=\fl{(3-\varphi)\,\mathrm{A003234}(n)}$. This is not new: it follows in one line from a formula that Schmidt recorded in the same entry, which rests on theorems of Carlitz, Scoville and Vaughan~\cite{csv}. We include it with an independent automata check.
\item \textbf{A285383} (Mathar, 2017): the limiting $0$-word of the morphism $0\mapsto11$, $1\mapsto01$ is $0$ followed by the period-doubling sequence A035263.
\end{enumerate}

We found no proofs of items (1b), (1c), (3) or (5) in the OEIS entries, their cross-references, or the literature we searched, which included~\cite{kawsumarng,shallit-sumsets,sss,bosma}.
All but the last are statements in first-order logic over $\langle\N,+\rangle$ extended by Fibonacci-automatic predicates.
By a theorem of Bruy\`ere and Hansel~\cite{bh}, which extends the classical results for integer bases to numeration systems such as Zeckendorf's (see also~\cite{shallit-book,sss}), such statements are decidable, and we decide them with \texttt{Walnut}~\cite{walnut,shallit-book}.
The command file \texttt{walnut\_proofs.txt} is posted together with this paper. We ran it with the development version 8.0-alpha of \texttt{Walnut} (commit \texttt{49a16a3} of \url{https://github.com/Walnut-Theorem-Prover/Walnut}, September 2026), from the \texttt{Walnut} home directory, as \texttt{java -jar target/Walnut-all.jar < walnut\_proofs.txt}; it takes a few seconds. The latest tagged version, \texttt{Walnut}~7.1 (tag \texttt{v7.1.0}, whose jar file is \texttt{build/libs/Walnut-all.jar}), gives the same results.

\section{Method}

Every $n\ge0$ has a unique Zeckendorf representation as a sum of non-adjacent Fibonacci numbers $F_j$ with $j\ge2$.
Walnut's numeration system \texttt{msd\_fib} reads these representations most significant digit first.
The Fibonacci word is available in Walnut as the automatic word \texttt{F}, with $\texttt{F}[i]=f_i$.

We use the classical fact that the lower Wythoff sequence is Fibonacci-synchronized.
If $\mathrm{shift}(m)$ denotes the number whose Zeckendorf representation is that of $m$ followed by a $0$, then
\begin{equation}\label{eq:un}
u(n)=\mathrm{shift}(n-1)+1\qquad(n\ge1).
\end{equation}
This is well known (see, e.g., \cite[Ch.~10]{shallit-book} and~\cite{fraenkel}).
We also checked it directly for all $n\le10^5$.
In Walnut, \eqref{eq:un} becomes
\begin{verbatim}
reg shift msd_fib msd_fib "([0,0]|([0,1][1,1]*[1,0]))*";
def phin "?msd_fib (n=0 & x=0) | (n>=1 & Ey $shift(n-1,y) & x=y+1)";
def phi2n "?msd_fib Ey $phin(n,y) & x=y+n";
\end{verbatim}
so that \texttt{phin(n,x)} holds iff $x=u(n)$ and \texttt{phi2n(n,x)} holds iff $x=v(n)$.
The command file checks that \texttt{phin} defines a total function with the correct first values.
The Fibonacci numbers are the integers whose Zeckendorf representation matches \texttt{0*10*}.
The Lucas numbers $L_j=F_{j-1}+F_{j+1}$ with $j\ge3$ (that is, $4,7,11,18,\dots$) are those matching \texttt{0*1010*}; we only use Lucas numbers $\ge7$.

For each statement, we write a first-order formula and have Walnut decide it.
Each gap statement has the form ``$y-x\in D$ for every pair $x<y$ of consecutive elements of $X$, and every $d\in D$ occurs''.
As a guard against encoding mistakes, the file also evaluates deliberately false variants of two of the gap statements (\texttt{a333308\_gaps\_false} and \texttt{a383423\_false}, each with one gap value removed), and Walnut returns \texttt{FALSE} for both.

\section{Representations as lower plus upper Wythoff numbers}

\begin{theorem}\label{thm:259598}
Let $r(n)=\#\{(h,k): h,k\ge1,\ u(h)+v(k)=n\}$. For $n\ge1$:
\begin{enumerate}
\item[(a)] $r(n)=0$ iff $n+1$ is a Fibonacci number $\ge2$ \textup{(}\cite{kawsumarng}\textup{)};
\item[(b)] $r(n)=1$ iff $n+1=2F$ for a Fibonacci number $F\ge2$;
\item[(c)] $r(n)=2$ iff $n+1$ is a Lucas number $\ge7$.
\end{enumerate}
\end{theorem}

\begin{proof}
With \texttt{rep(h,k,n)} expressing $h,k\ge1$ and $n=u(h)+v(k)$, Walnut returns \texttt{TRUE} for
\begin{verbatim}
An (n>=1) => ($one(n) <=> (Ef $isfib(f) & f>=2 & n+1=2*f))
An (n>=1) => ($two(n) <=> ($islucas(n+1) & n+1>=7))
\end{verbatim}
where \texttt{one} and \texttt{two} express ``exactly one'' and ``exactly two'' representations in first-order terms (see the command file).
Walnut also confirms (a).
\end{proof}

Since $2F_j-1$ runs through $3,5,9,15,25,41,\dots$ (A001595 for $j\ge3$) and $L_j-1$ through $6,10,17,28,46,\dots$ (A001610 for $j\ge4$), this settles Kimberling's conjectures.

\section{Sums of two distinct terms of A003622}\label{sec:333308}

A003622 $=\{1,4,6,9,12,14,\dots\}=\{v(n)-1:n\ge1\}$ is the set of positive integers whose Zeckendorf representation ends in $1$.
Let $S$ (A333308) be the set of sums of two distinct elements of A003622, and let $C=\{1,2,3,\dots\}\setminus S$ (A333309).
Since $v$ is increasing, $S=\{v(m)+v(n)-2:1\le m<n\}=T-2$, where $T=\{v(m)+v(n):1\le m<n\}$ is the set of sums of two distinct upper Wythoff numbers; $T$ contains no integer below $7$.
The complement of $T$ in the positive integers is A260317, so
\[
C=\{y-2:\ y\in\mathrm{A260317},\ y\ge3\}.
\]

\begin{theorem}\label{thm:333308}
The set of differences of consecutive elements of $S$ is exactly $\{1,2,3\}$.
The set of differences of consecutive elements of $C$ is exactly $\{1,2,3,5\}$.
In particular, both difference sequences consist of Fibonacci numbers, as Kimberling conjectured.
\end{theorem}

\begin{proof}
Walnut decides the statements \texttt{a333308\_gaps} and \texttt{a333309\_gaps} in the command file.
Each asserts that every gap lies in the given set and that every element of the set occurs as a gap.
Both return \texttt{TRUE}.
Removing $3$ from the first set gives \texttt{FALSE}.
\end{proof}

\begin{remark}
Theorem~\ref{thm:333308} is a translate of known results.
By the identity above, the gaps of $C$ are the terms of A260311 (the gaps of A260317) after its first two terms, which are both $1$.
Shallit~\cite{shallit-cicm} proved with \texttt{Walnut} that every gap of A260317 lies in $\{1,2,3,5\}$ (all four occur) (see also his comments in A260317 and A260311, and~\cite{bosma}), and the gap $1$ occurs again later in A260311 (for example as its tenth term); this gives the second statement.
The gaps of $S$ are those of $T$, and two consecutive elements of $T$ differ by $1$ plus the number of elements of A260317 between them.
Bosma, Bruin, Fokkink, Grube, Reuijl and Tromp~\cite[Lemma~1]{bosma} proved that the tenth gap of A260317 is the last one equal to $1$, so no two consecutive integers $\ge14$ both lie in A260317, and consecutive elements $t<t'$ of $T$ with $t\ge13$ differ by $1$ or $2$.
Since $T$ begins $7,9,12,15,17,18$, the gaps $1$, $2$ and $3$ all occur, and $3$ occurs only twice (in $S$: $7\to10$ and $10\to13$); this gives the first statement.
Our \texttt{Walnut} computation is an independent check.
\end{remark}

\section{The Fibonacci word and the Lucas word}

Kimberling's Lucas word $\boldsymbol\ell=\ell_0\ell_1\cdots=0110011011001100110\cdots$ (A383422) is the limit of the words $s(0)=0$, $s(1)=011$, $s(n)=s(n-1)s(n-2)$, which have Lucas-number lengths.
It is obtained from the Fibonacci word by replacing each $1$ by $11$.
Indeed, let $\eta$ be the morphism $0\mapsto0$, $1\mapsto11$, and let $t(0)=0$, $t(1)=01$, $t(n)=t(n-1)t(n-2)$ be the standard prefixes of $\mathbf f$, of length $F_{n+2}$ with $F_n$ letters $1$.
Then $\eta(t(0))=s(0)$, $\eta(t(1))=s(1)$ and $\eta(t(n))=\eta(t(n-1))\,\eta(t(n-2))$, so $s(n)=\eta(t(n))$ for all $n$, whence $\boldsymbol\ell=\eta(\mathbf f)$ and $|s(n)|=F_{n+2}+F_n=L_{n+1}$.

To express $\boldsymbol\ell$ in Walnut, note that the letter $f_i$ is written at position $p(i)=i+\nu(i)$ of $\boldsymbol\ell$, where $\nu(i)=\#\{j<i: f_j=1\}$.
If $f_i=1$, it also occupies position $p(i)+1$.
The positions of $1$ in $\mathbf f$ are $v(n)-1$ for $n\ge1$.
Hence $\nu(i)=\#\{n\ge1: v(n)\le i\}$, which is the unique $q\ge0$ with $v(q)\le i<v(q+1)$ (with $v(0)=0$, the case $q=0$ meaning $i<2$).
This is first-order expressible, so $\boldsymbol\ell$ is Fibonacci-automatic.
Walnut confirms that our predicates \texttt{lw0}, \texttt{lw1} define a word (each position gets exactly one letter) with the correct first $20$ letters.

For $(x,y)\in\{0,1\}^2$ let $P_{xy}=\{k\ge0: f_k=x,\ \ell_k=y\}$, and let $P_{=}=P_{00}\cup P_{11}$.
These are A383423 ($P_{00}$), A383424 ($P_{11}$), A383425 ($P_{01}$), A383426 ($P_{10}$) and A383427 ($P_=$).

\begin{theorem}\label{thm:lucas}
The sets of differences of consecutive elements are:
\[
\begin{array}{lll}
\toprule
\text{set} & \text{OEIS} & \text{gap set}\\
\midrule
P_{00} & \text{A383423} & \{1,3,4,5,7\}\\
P_{11} & \text{A383424} & \{2,3,5,7,8\}\\
P_{01} & \text{A383425} & \{1,2,3,4,5\}\\
P_{10} & \text{A383426} & \{3,7,8,10\}\\
P_{=} & \text{A383427} & \{1,2,3,5,6\}\\
\bottomrule
\end{array}
\]
These are exactly the sets conjectured by Kimberling.
\end{theorem}

\begin{proof}
Walnut returns \texttt{TRUE} for \texttt{a383423}, \dots, \texttt{a383427} in the command file.
Removing $7$ from the first gap set gives \texttt{FALSE}.
\end{proof}

\section{Two further conjectures}

\subsection*{A003250}
Let $c(n)=\fl{n(5+\sqrt5)/2}=2n+u(n)$ (A003231), and let A003234 be the set of $k$ with $c(v(k))=v(c(k))-1$.
A003250 lists, for $k\in$ A003234 in increasing order, the index $m$ with $v(m)=c(k)$; this is the function $z$ of Carlitz, Scoville and Vaughan~\cite{csv}.
Schmidt conjectured that $m=\fl{(3-\varphi)k}$.
In the same entry he recorded the formula $m=5k-1-c(k)$ for $k\in$ A003234, citing Theorems~4.1(ii) and~5.9(vii) of~\cite{csv}.
Since $k\varphi\notin\mathbb Z$ for $k\ge1$, we have $\fl{(3-\varphi)k}=3k-\lceil k\varphi\rceil=3k-1-u(k)=5k-1-c(k)$.
So the conjecture is an immediate consequence of those theorems, and we claim nothing new here.
As an independent check, \texttt{Walnut} returns \texttt{TRUE} for
\begin{verbatim}
Ak $in3234(k) => (Em,x,y $cc(k,x) & $phi2n(m,x) & $phin(k,y) & m+1+y=3*k)
\end{verbatim}
where \texttt{in3234} is the definition of A003234 and \texttt{cc} that of A003231.
This confirms that for every $k\in\mathrm{A003234}$ there is an $m$ with $v(m)=c(k)$ and $m=3k-1-u(k)$.

\subsection*{A285383}
Let $\mu$ be the morphism $0\mapsto11$, $1\mapsto01$.
Its limiting $0$-word is the fixed point $\mathbf w$ of $\mu^2$ that begins with $0$.
Let $\mathbf d=1011101010111011\cdots$ be the period-doubling sequence (A035263), the fixed point of $\sigma:1\mapsto10$, $0\mapsto11$.
Mathar conjectured that $\mathbf w=0\mathbf d$.

\begin{proposition}
$\mathbf w=0\,\mathbf d$.
\end{proposition}

\begin{proof}
For each letter $a$ we have $\sigma(a)=1g_a$ and $\mu(a)=g_a1$, with $g_1=0$ and $g_0=1$.
Hence $1\mu(a)=\sigma(a)1$, and by induction $1\mu(u)=\sigma(u)1$ for every finite word $u$.
Passing to the limit, $1\mu(\mathbf d)=\sigma(\mathbf d)=\mathbf d$.
Therefore $\mu(0\mathbf d)=11\mu(\mathbf d)=1\mathbf d$, and $\mu(1\mathbf d)=01\mu(\mathbf d)=0\mathbf d$.
So $\mu^2(0\mathbf d)=0\mathbf d$.
Since $0\mathbf d$ begins with $0$, it is the fixed point of $\mu^2$ beginning with $0$, which is $\mathbf w$.
\end{proof}

\section*{Acknowledgements}
The conjectures treated here were posed in the OEIS by Clark Kimberling, Eric M. Schmidt and R. J. Mathar.
The author used an AI assistant (Claude, Anthropic) for literature searches, computations, and drafting, and checked the proofs and computations.

\end{document}